\RequirePackage[l2tabu, orthodox]{nag}

\documentclass[12pt,reqno]{amsart}
\usepackage{fullpage,url,amssymb,enumerate,colonequals}
\usepackage{mathrsfs} 

\usepackage[OT2,T1]{fontenc}
\DeclareSymbolFont{cyrletters}{OT2}{wncyr}{m}{n}
\DeclareMathSymbol{\Sha}{\mathalpha}{cyrletters}{"58}

\usepackage{color}

\newcommand{\defi}[1]{\textsf{#1}} 

\newcommand{\Aff}{\mathbb{A}}

\newcommand{\PP}{\mathbb{P}}

\newcommand{\Z}{\mathbb{Z}}

\newcommand{\Ahat}{{\widehat{A}}}
\newcommand{\Rhat}{{\widehat{R}}}

\newcommand{\ii}{\mathbf{i}}

\newcommand{\pp}{\mathfrak{p}}
\newcommand{\mm}{\mathfrak{m}}

\newcommand{\calH}{\mathcal{H}}

\newcommand{\OO}{\mathscr{O}}
\newcommand{\OOhat}{\widehat{\OO}}

\DeclareMathOperator{\Char}{char}

\DeclareMathOperator{\Det}{Det}

\DeclareMathOperator{\Frac}{Frac}

\DeclareMathOperator{\Gr}{Gr}

\DeclareMathOperator{\mult}{mult}

\DeclareMathOperator{\Proj}{Proj}

\DeclareMathOperator{\rank}{rank}

\DeclareMathOperator{\Res}{Res}

\DeclareMathOperator{\Span}{Span}
\DeclareMathOperator{\Spec}{Spec}

\DeclareMathOperator{\vmin}{vmin}

\newcommand{\red}{{\operatorname{red}}}
\newcommand{\sing}{{\operatorname{sing}}}
\newcommand{\smooth}{{\operatorname{smooth}}}

\newcommand{\M}{\operatorname{M}}

\newcommand{\del}{\partial}
\newcommand{\directsum}{\oplus} 
\newcommand{\injects}{\hookrightarrow}
\newcommand{\intersect}{\cap} 
\newcommand{\isom}{\simeq}

\newcommand{\surjects}{\twoheadrightarrow}

\newcommand{\union}{\cup} 

\newcommand{\Dfinite}{D_{\textup{finite}}}
\newcommand{\DfiniteA}{D_{\textup{finite},A}}
\newcommand{\DfiniteR}{D_{\textup{finite},R}}

\newcommand{\negspace}{\mathchoice{\hspace*{-1.5pt}}{\hspace*{-1.5pt}}{\hspace*{-1.1pt}}{\hspace*{-0.8pt}}}
\newcommand{\pws}[1]{[\negspace[#1]\negspace]} 
\newcommand{\ideal}[1]{\langle #1 \rangle} 

\numberwithin{equation}{section}

\newtheorem{theorem}[equation]{Theorem}
\newtheorem*{theorem1.1}{Theorem~1.1}
\newtheorem{lemma}[equation]{Lemma}
\newtheorem{corollary}[equation]{Corollary}
\newtheorem{proposition}[equation]{Proposition}

\theoremstyle{definition}
\newtheorem{definition}[equation]{Definition}

\theoremstyle{remark}
\newtheorem{remark}[equation]{Remark}

\numberwithin{equation}{section}

\makeatletter
\g@addto@macro\bfseries{\boldmath} 
\makeatother

\usepackage{microtype}

\definecolor{darkgreen}{rgb}{0,0.5,0}

\usepackage[
	colorlinks, citecolor=darkgreen,
	pagebackref,
	pdfauthor={Bjorn Poonen and Michael Stoll},
	pdftitle={The valuation of the discriminant of a hypersurface},
]{hyperref}
\usepackage[alphabetic,backrefs,lite,nobysame]{amsrefs} 

\begin{document}

\title{The valuation of the discriminant of a hypersurface}
\author{Bjorn Poonen}
\address{Department of Mathematics,
         Massachusetts Institute of Technology,
         Cambridge, MA 02139-4307, USA}
\email{poonen@math.mit.edu}
\urladdr{\url{http://math.mit.edu/~poonen/}}

\author{Michael Stoll}
\address{Mathematisches Institut,
         Universit\"at Bayreuth,
         95440 Bayreuth, Germany}
\email{Michael.Stoll@uni-bayreuth.de}
\urladdr{\url{http://www.mathe2.uni-bayreuth.de/stoll/}}

\thanks{B.P.\ was supported in part by National Science Foundation grants DMS-1601946 and DMS-2101040 and Simons Foundation grants \#402472 (to Bjorn Poonen) and \#550033.}
\date{October 16, 2025}

\begin{abstract}
Let $R$ be a discrete valuation ring,
with valuation $v \colon R \surjects \Z_{\ge 0} \union \{\infty\}$
and residue field $k$.
Let $H$ be a hypersurface $\Proj(R[x_0,\ldots,x_n]/\ideal{f})$.
Let $H_k$ be the special fiber,
and let $(H_k)_{\sing}$ be its singular subscheme.
Let $\Delta(f)$ be the discriminant of $f$.
We use Zariski's main theorem and degeneration arguments
to prove that $v(\Delta(f))=1$
if and only if $H$ is regular
and $(H_k)_{\sing}$ consists of a nondegenerate double point over $k$.
We also give lower bounds on $v(\Delta(f))$
when $H_k$ has multiple singularities or a positive-dimensional singularity.
\end{abstract}

\maketitle


\section{Introduction}
\label{S:introduction}

Throughout this paper, $R$ denotes a discrete valuation ring,
with valuation $v \colon R \surjects \Z_{\ge 0} \union \{\infty\}$,
maximal ideal $\mm = (\pi)$,
and residue field $k$
(except in a few places where $k$ denotes an arbitrary field).

Let $E \subset \PP^2_R$ be defined by a Weierstrass equation,
with generic fiber an elliptic curve.
If the discriminant of the equation has valuation~$1$,
then $E$ is regular and the singular locus of its special fiber
consists of a node; this follows from Tate's algorithm \cite{Tate1975},
for example; see also \cite{SilvermanATAEC}*{Lemma~IV.9.5(a)}.
Our first theorem (Theorem~\ref{T:valuation 1})
generalizes this to hypersurfaces of arbitrary degree and dimension
(terminology will be explained later):

\begin{theorem}
\label{T:valuation 1}
Let $f \in R[x_0,\dots,x_n]$ be a homogeneous polynomial.
Let $\Delta(f)$ be its discriminant.
Let $H = \Proj(R[x_0,\dots,x_n]/\ideal{f})$.
Then the following are equivalent:
\begin{enumerate}[\upshape (i)]
\item $v(\Delta(f))=1$;
\item $H$ is regular,
and $(H_k)_{\sing}$ consists of a single nondegenerate double point in $H(k)$.
\end{enumerate}
\end{theorem}

For hypersurfaces with more than one singularity, we have the following:

\begin{theorem}
  \label{T:2}
  Let $f$ and $H$ be as in Theorem~\ref{T:valuation 1}.  
  \begin{enumerate}[\upshape (a)]
  \item \label{I:2a} If $(H_k)_{\sing}$ consists of $r$ closed points, $v(\Delta(f)) \ge r$ $($Theorem~\ref{T:r isolated singularities}$)$.
  \item \label{I:2b}
    We have $v(\Delta(f)) \ge \dim (H_k)_{\sing} + 1$ $($Theorem~\ref{T:positive-dimensional singularity}\eqref{I:positive-dimensional singularity,a}$)$.
  \item \label{I:2c}
    If $\dim (H_k)_{\sing} \ge 1$, then $v(\Delta(f)) \ge \lfloor (\deg f-1)/2 \rfloor$ $($Theorem~\ref{T:positive-dimensional singularity}\eqref{I:positive-dimensional singularity,b}$)$.
  \end{enumerate}
\end{theorem}

To prove~\eqref{I:2c}, we show that $H_k$ is a limit
of hypersurfaces whose singular subscheme is finite but with many points,
and we combine this and an argument using restriction of scalars
in the equal characteristic case and the Greenberg functor in the mixed
characteristic case.

The paper is organized as follows.
Section~\ref{S:discriminant} defines the discriminant $\Delta$ of a projective hypersurface $f=0$ and proves some basic properties of it.
Section~\ref{S:quadratic forms} describes quadratic forms over a discrete valuation ring, and computes their discriminants.
Section~\ref{S:nondegenerate double points} defines nondegenerate and ordinary double points.
Section~\ref{S:general singular locus} adapts the proof of the Bertini smoothness theorem to prove that the singular locus of the general singular hypersurface over a field consists of a single ordinary double point.
Our proofs require some ingredients from commutative algebra, provided in Section~\ref{S:CA}.
Section~\ref{S:several} proves Theorem~\ref{T:2}\eqref{I:2a} = Theorem~\ref{T:r isolated singularities}.
Section~\ref{S:valuations} analyzes the minimum valuation of values of a multivariable polynomial on a residue disk,
and Section~\ref{S:vminDelta} applies this analysis to $\Delta$, viewed as a polynomial in the coefficients of $f$.
Finally, Section~\ref{S:valuation 1} proves Theorem~\ref{T:valuation 1},
and Section~\ref{S:positive-dimensional} proves the rest of Theorem~\ref{T:2}.
  
\section{The discriminant}
\label{S:discriminant}

Fix $n \ge 1$ and $d \ge 2$.
Let $A$ be a ring.
Let $A[x_0,\ldots,x_n]_d$ be the set of homogeneous polynomials of degree $d$.
Let $f \in A[x_0,\ldots,x_n]_d$.
Let $H=H_f = \Proj(k[x_0,\ldots,x_n]/\ideal{f}) \subset \PP^n_A$.
Define the \defi{relative singular subscheme} $H_{\sing}$
as the closed subscheme of $\PP^n_A$ defined by $f=\del f/\del x_0 = \cdots = \del f/\del x_n = 0$.
Its complement $H_{\smooth} \colonequals H - H_{\sing}$ is the locus of points at which $H \to \Spec A$ is smooth of relative dimension $n-1$.

Let $x^{\ii}$ range over the $N \colonequals \binom{n+d}{n}$
monomials in $\Z[x_0,\ldots,x_n]_d$,
and let $a_{\ii}$ be independent indeterminates in $\Z[\{ a_{\ii} \}]$,
so $F \colonequals \sum_{\ii} a_{\ii} x^{\ii}$
is the generic degree~$d$ homogeneous polynomial in $x_0,\ldots,x_n$.
Then the affine space $\Aff^N \colonequals \Spec \Z[\{ a_{\ii} \}]$
may be viewed as a moduli space for hypersurfaces
(one could also remove the origin,
or projectivize as in \cite{Saito2012}*{\S2.4}).
Specializing the previous paragraph to $f=F$ and $A=\Z[\{ a_{\ii} \}]$
gives the \defi{universal hypersurface} $\calH \subset \PP^n \times \Aff^N$
and its relative singular subscheme $\calH_{\sing}$,
relative to the second projection $\phi \colon \calH \to \Aff^N$.

The first projection makes $\calH_{\sing} \to \PP^n$
a rank $N-n-1$ vector bundle since the equations
$F=\del F/\del x_0 = \cdots = \del F/\del x_n = 0$
are linear in the $a_{\ii}$
and independent above each point of $\PP^n$
except for the Euler relation $d \cdot F = \sum x_i (\del F/\del x_i)$.
Thus $\calH_{\sing}$ is integral
and smooth of relative dimension $N-1$ over $\Z$.
Since $\phi$ is proper,
the image $D \colonequals \phi(\calH_{\sing}) \subset \Aff^N$
is a closed integral subscheme;
$D$ is the locus parametrizing \emph{singular} hypersurfaces.
In fact, $D \subset \Aff^N$ is a divisor
and the restriction $\calH_{\sing} \to D$ of $\phi$ is birational
(see Proposition~\ref{P:birational} below);
this is a Bertini-type statement saying essentially
that among hypersurfaces singular at a point,
most have singular subscheme consisting of just that point.
Thus $D \subset \Aff^N$ is the zero locus of some polynomial
$\Delta \in \Z[\{a_{\ii}\}]$
determined up to a unit, i.e., up to sign;
$\Delta$ is called the \defi{discriminant}.
(See \cites{Gelfand-Kapranov-Zelevinsky2008,Demazure2012,Saito2012}
for other descriptions of $\Delta$.)
By definition, if the $a_{\ii}$ are specialized to elements
of a field $k$, the resulting hypersurface in $\PP^n_k$
is singular (not smooth of dimension $n-1$)
if and only if $\Delta$ specializes to~$0$ in~$k$.
It is a classical fact that the polynomial $\Delta$ is homogeneous of degree
$(n+1) (d-1)^n$ in the $N$ variables \cite{Eisenbud-Harris2016}*{Proposition~7.4}.

\section{Quadratic forms}
\label{S:quadratic forms}

\begin{proposition}
\label{P:discriminant of quadratic form}
Suppose that $d=2$.
Let $\Det = \det (\del^2 F/\del x_i \del x_j ) \in \Z[\{a_\ii\}]$.
If $n$ is odd, then $\Delta = \pm \Det$.
If $n$ is even, then $\Delta = \pm \Det/2$.
\end{proposition}

\begin{proof}
This is well known, except perhaps the power of $2$,
which can be determined by evaluating $\Det$
for a quadratic form defining a smooth quadric over $\Z$,
since $\Delta = \pm 1$ for such a form.
Use $x_0 x_1 + \cdots + x_{n-1} x_n$ if $n$ is odd,
and $x_0 x_1 + \cdots + x_{n-2} x_{n-1} + x_n^2$ if $n$ is even.
\end{proof}

Let $R$ be a discrete valuation ring or field.
A \defi{quadratic space} over $R$
is a pair $(M,q)$
where $M$ is a finite-rank free $R$-module (since $R$ is a PID, we do not to say the word \emph{projective}),
and $q \colon M \to R$ is such that if $e_1,\ldots,e_n$ is a basis of $M$,
then $q(x_1 e_1 + \cdots + x_n e_n)$ is given by a polynomial in $R[x_1,\ldots,x_n]_2$.
A \defi{symmetric bilinear space} over $R$
is a pair $(M,\beta)$
where $M$ is as before
and $\beta \colon M \times M \to R$ is a symmetric $R$-bilinear pairing.
Given $q$, define $\beta(x,y) \colonequals q(x+y)-q(x)-q(y)$;
in this way, every quadratic space has an associated symmetric bilinear space.

\begin{proposition}
\label{P:general quadratic form}
Let $k$ be an algebraically closed field.
A general quadratic form $q \in k[x_1,\ldots,x_n]_2$ is equivalent $($via a linear change of variables$)$ to
\[
	\begin{cases}
	x_1^2+\cdots+x_n^2, & \textup{ if $\Char k \ne 2$;} \\
	x_1 x_2 + x_3 x_4 + \cdots + x_{n-1} x_n, & \textup{ if $\Char k = 2$ and $n$ is even;} \\
	x_1 x_2 + x_3 x_4 + \cdots + x_{n-2} x_{n-1} + x_n^2, & \textup{ if $\Char k = 2$ and $n$ is odd.} \\
        \end{cases}
\]
(\emph{General} means that there is a dense open subset $U$ of the coefficient space such that the statement holds for $q$ corresponding to a point of $U$.)
\end{proposition}

\begin{proof}
The associated symmetric bilinear space may be identified with the matrix
$M \colonequals \left( \del^2 q/\del x_i \del x_j \right)$.

First suppose that $\Char k \ne 2$.
For the general $q$, the symmetric matrix $M$ has rank $n$ (since $\det M \ne 0$ defines a nonempty open subset), and after a change of variables $M$ is diagonal and $q$ is $x_1^2+\cdots+x_n^2$.

Next suppose that $\Char k=2$.
Then $M$ is symplectic, so $\rank M$ is even.
If $n$ is even, then for the general $q$, the matrix $M$ is of rank $n$,
and after a change of variable to put $M$ in standard form,
$q$ is $x_1 x_2 + \cdots + x_{n-1} x_n$.
Now suppose that $n$ is odd.
For the general $q$, the matrix $M$ is of rank $n-1$.
After a change of variables,
$q$ is $x_1 x_2 + \cdots + x_{n-2} x_{n-1} + \ell^2$, for some linear form $\ell$.
By adding a multiple of $x_1$ to $x_2$, we may assume that $x_1$ does not appear in $\ell$.
Similarly, we can eliminate $x_2,\ldots,x_{n-1}$ from $\ell$,
so $\ell$ is a multiple of $x_n$.
For the general $q$, we may assume that $\ell$ is a \emph{nonzero} multiple of $x_n$.
By scaling, we may assume that $\ell=x_n$.
Now $q= x_1 x_2 + x_3 x_4 + \cdots x_{n-2} x_{n-1} + x_n^2$.
\end{proof}

\begin{proposition}
\label{P:quadratic form}
Let $R$ be a discrete valuation ring.
\begin{enumerate}[\upshape (a)]
\item
\label{I:symmetric bilinear space}
Each symmetric bilinear space over $R$
is an orthogonal direct sum of spaces of rank $1$ or~$2$.
\item
\label{I:quadratic form}
Every quadratic form $f(x_0,\ldots,x_n)$ over $R$
is equivalent to one of the form
\[
	\sum_{i=1}^I (a_i x_i^2 + b_i x_i y_i + c_i y_i^2)
	+ \sum_{j=1}^J d_j z_j^2
\]
with $2I+J=n+1$ and $a_i,b_i,c_i,d_j \in R$.
\item
\label{I:inequality for quadratic form}
Let $f$ be as in~\eqref{I:quadratic form}.
Let $H = \Proj(R[x_0,\ldots,x_n]/\ideal{f})$.
If $H_k$ is smooth, then $v(\Delta(f))=0$.
Otherwise, $v(\Delta(f)) \ge \dim \, (H_k)_{\sing} + 1$.
\end{enumerate}
\end{proposition}

\begin{proof}\hfill
\begin{enumerate}[\upshape (a)]
\item
(We paraphrase an argument of Jean-Pierre Tignol
adapted from the proof of \cite{Verstraete2019}*{Proposition~4.10}.)
Let $(M,\beta)$ be a nonzero symmetric bilinear space.
We may assume that $\beta \ne 0$.
By dividing $\beta$ by a nonzero element of $R$,
we may assume that $\beta(M,M) \not\subset \mm$.
We claim that there exists a free $R$-module $N$ of rank $1$ or $2$
with a homomorphism $N \to M$
such that $\beta$ induces a \defi{regular} pairing on $N$
(i.e., the composition $N \to M \stackrel{\beta}\to M^\vee \to N^\vee$
is an isomorphism);
then $N \to M$ is injective, and $M$ is the orthogonal direct sum
of $N$ and $N^\perp \colonequals \ker(M \to N^\vee)$,
so we are done by induction on $\rank(M)$.

If there exists $e \in M$ with $\beta(e,e) \in R^\times$ a unit,
then let $N=Re$.
Otherwise, choose $c,d \in M$ with $\beta(c,d) \in R^\times$
and let $N=Rc \directsum Rd$;
the induced pairing is regular
since its matrix is invertible, being congruent mod $\pi$ to
$\begin{pmatrix} 0 & \beta(c,d) \\ \beta(c,d) & 0 \end{pmatrix}$.
\item
Decomposing a quadratic space is equivalent
to decomposing the associated symmetric bilinear space,
even if $\Char k=2$.
\item
The case where $H_k$ is smooth is true by the definition of discriminant,
so suppose that $(H_k)_{\sing} \ne \emptyset$.
  
First suppose $\Char k \ne 2$.
Then $f$ is equivalent to $\sum a_i x_i^2$ for some $a_i \in R$,
and
\[
	\dim \, (H_k)_{\sing} = \#\{i : v(a_i) \ge 1 \} - 1 \le v(\Det(f)) - 1 = v(\Delta(f)) - 1,
\]
by Proposition~\ref{P:discriminant of quadratic form}.

Now suppose $\Char k = 2$.
Let $I_0 = \#\{i : v(b_i) = 0\}$ and $I_1 = \#\{i : v(b_i) \ge 1\}$.
Let $J_0 = \#\{j : v(d_j) = 0\}$ and $J_1 = \#\{j : v(d_j) \ge 1\}$.
If $n$ is odd, let $J' \colonequals J$.
If $n$ is even, let $J' \colonequals J-1$.
In both cases $J' \ge 0$ (if $n$ is even, then $J$ is odd).
The common zero locus in $\PP^n_k$
of the polynomials $\del f/\del x_i$ and $\del f/\del y_i$
for $i \in I_0$ is of dimension $n-2I_0$,
and including the condition $f=0$ drops the dimension by $1$ more
if $J_0 \ge 1$.
Thus $\dim \, (H_k)_{\sing} \le n-2I_0$, with strict inequality if $J_0 \ge 1$.
On the other hand,
$v(4a_i c_i -b_i^2) \ge 2$ whenever $v(b_i) \ge 1$,
and $v(2d_j) \ge v(2) + v(d_j)$ for all $j$,
so Proposition~\ref{P:discriminant of quadratic form} implies
\begin{align*}
	v(\Delta(f))
	&\ge 2I_1 + J' v(2) + J_1 \\
	&= (n - 2I_0) + J' v(2) - J_0 + 1 \\
	&\ge \dim \, (H_k)_{\sing}  + J' v(2) - J_0 + 1.
\end{align*}
If $J_0 \ge 1$, then the inequality above is strict
and $J' v(2) \ge (J_0-1) v(2) \ge J_0-1$,
so $v(\Delta(f)) \ge \dim \, (H_k)_{\sing} + 1$.
If $J_0=0$, then instead use $J' v(2) \ge 0$
to again get $v(\Delta(f)) \ge \dim \, (H_k)_{\sing} + 1$.\qedhere
\end{enumerate}
\end{proof}

\section{Nondegenerate double points and ordinary double points}
\label{S:nondegenerate double points}

\begin{definition}[\cite{SGA7.1}*{VI.6}]
  \label{D:nondegenerate double point}
Let $k$ be a field.
Let $X$ be a finite-type $k$-scheme.
A $k$-point $Q \in X$ is called a \defi{nondegenerate double point}
(or \defi{nondegenerate quadratic point})
if there exist $n \ge 1$ and $f \in k\pws{x_1,\ldots,x_n}$
such that there is an isomorphism of complete $k$-algebras
$\widehat{\OO}_{X,Q} \isom k\pws{x_1,\ldots,x_n}/\ideal{f}$
and an equality of ideals
$\ideal{\del f/\del x_1,\ldots,\del f/\del x_n} = \ideal{x_1,\ldots,x_n}$.
\end{definition}

\begin{remark}
\label{R:nondegenerate double point}
The ideal equality is equivalent to saying that
$Q$ is an isolated reduced point of the singular subscheme $X_{\sing}$.
\end{remark}

\begin{remark}
  \label{R:nondegenerate double point on affine hypersurface}
If $X$ is an affine hypersurface in $\Aff^n_k$ given by the equation $h(x_1,\ldots,x_n) = 0$, then
a singular point~$Q$ on~$X$ is a nondegenerate double point if and only
if $\det\left(\frac{\del^2 h}{\del x_i \del x_j}\right)$ does not vanish
at~$Q$.
\end{remark}

\begin{remark}
\label{R:explicit nondegenerate double point}
Suppose that $n$ and $f$ as in Definition~\ref{D:nondegenerate double point}
exist.
Then $f$ can be taken to be a quadratic form \cite{SGA7.1}*{VI.6.1}.
If, moreover, $k$ is algebraically closed, then
\begin{itemize}
\item if $\Char k \ne 2$, then one can take $f \colonequals x_1^2+\ldots+x_n^2$;
\item if $\Char k = 2$, then $n$ must be even and one can take
$f \colonequals x_1 x_2 + x_3 x_4 + \cdots+ x_{n-1} x_n$.
\end{itemize}
\end{remark}

\begin{definition}[\cite{SGA7.1}*{Definition~VI.6.6}]
There is also the notion of \defi{ordinary double point},
which is the same except when $\Char k=2$ and the local dimension $n$ of $X$ at $Q$ is odd.
In that case, nondegeneracy is impossible
so one calls a singularity an ordinary double point  
if and only if it is analytically equivalent over an algebraic closure
to that defined by $x_1 x_2 + \cdots+ x_{n-2} x_{n-1} + x_n^2$.
\end{definition}

\section{The general singular locus}
\label{S:general singular locus}

\begin{proposition}
\label{P:general singular locus}
Fix $n \ge 1$ and $d \ge 2$ and an algebraically closed field $k$.
For general $f \in k[x_0,\ldots,x_n]_d$ with $H_f$ singular
$($that is, $f$ corresponding to a general point of $D(k)$$)$,
the hypersurface $H_f$ has a unique singularity 
and it is an ordinary double point.
\end{proposition}

\begin{proof}
\emph{Case $d=2$.}
Let $M=\left( \del^2 f/\del x_i \del x_j \right) \in \M_{n+1}(k)$.

First suppose that $\Char k \ne 2$.
For the general $f$, the symmetric matrix $M$ has rank $n$ (rank $\ge n$ is an open condition, and rank $n+1$ would mean that $H_f$ is smooth), and after a change of variable it is diagonal and $f$ is $x_1^2+\cdots+x_n^2$, and $(H_f)_{\sing}$ is the single reduced point $(1:0:\cdots:0)$.

Next suppose that $\Char k=2$.
Then $M$ is symplectic, so $\rank M$ is even.
If $n$ is even, then for the general $f$, the matrix $M$ is of rank $n$,
and after a change of variable to put $M$ in standard form,
$f$ is $x_1 x_2 + \cdots + x_{n-1} x_n$, 
and $(H_f)_{\sing}$ is the single reduced point $(1:0:\cdots:0)$.
If $n$ is odd, then for the general $f$, the matrix $M$ is of rank $n-1$,
and after a change of variable,
$f$ is $x_1 x_2 + \cdots + x_{n-2} x_{n-1} + x_n^2$, 
and $(H_f)_{\sing}$ is a nonreduced degree~$2$ scheme supported at $(1:0:\cdots:0)$.

In all these cases, the unique point of $(H_f)_{\sing}$ is an ordinary double point of $H_f$ (and it is even nondegenerate, except when $\Char k=2$ and $n$ is odd).

\medskip

\emph{Case $d \ge 3$.}
Let $(\PP^n \times \PP^n)'$ be the locus of pairs of points $(P,Q) \in \PP^n \times \PP^n$ with $P \ne Q$.
Let $I$ be the locus of $(f,P,Q) \in \Aff^N \times (\PP^n \times \PP^n)'$ such that $H_f$ is singular at both $P$ and $Q$.
The fibers of $I \to (\PP^n \times \PP^n)'$ are linear subspaces of codimension $2n+2$ in $\Aff^N$ since we may assume $P=(1:0:\cdots:0)$ and $Q=(0:\cdots:0:1)$, in which case $H_f$ is singular at $P$ and $Q$ if and only if the coefficients of $x_0^{d-1} x_i$ and $x_n^{d-1} x_i$ for $i=0,\ldots,n$ all vanish.
Thus $\dim I = (N-(2n+2)) + \dim (\PP^n \times \PP^n)' = N-2$, so $I$ does not dominate the $(N-1)$-dimensional locus $D \subset \Aff^N$ corresponding to $f$ with $H_f$ singular.

\medskip

Thus for general $f$ with $H_f$ singular, $H_f$ has only one singularity,
which we may assume is $P \colonequals (1:0:\cdots:0)$.
Proposition~\ref{P:general quadratic form} applied to the degree~$2$ Taylor polynomial at $P$ of a dehomogenization of $f$ shows that for general $f$, the singularity is an ordinary double point.
\end{proof}

We use subscripts to denote base change: e.g.,
$D_A \colonequals D \times_{\Spec \Z} \Spec A$ 
and $\calH_{\sing,A} \colonequals \calH_{\sing} \times_{\Spec \Z} \Spec A \isom (\calH_A)_{\sing}$ for any ring $A$.
For an irreducible scheme $X$, let $\kappa(X)$ be the function field of the integral scheme $X_{\red}$.
Recall that $\phi \colon \calH \injects \PP^n \times \Aff^N \surjects \Aff^N$ was the second projection.
Restricting $\phi$ yields a proper surjective morphism $\varphi \colon \calH_{\sing} \to D$.

\begin{proposition}
\label{P:birational}
The morphism $\varphi \colon \calH_{\sing} \to D$ is birational.
The same holds after base change to any integral domain $A$,
except when $\Char A=2$ and $n$ is odd, in which case $\kappa(\calH_{\sing,A})$ is purely inseparable of degree~$2$ over $\kappa(D_A)$.
\end{proposition}

\begin{proof}
We may assume that $A$ is a field $k$, and that $k$ is algebraically closed.
The result follows from \cite{Saito2012}*{Proposition~2.12}, except when $\Char k=2$ and $n$ is odd.  
We will reprove those cases and prove the missing cases.

For a general $f \in D(k)$, Proposition~\ref{P:general singular locus} implies that $H_f$ has an ordinary double point, so $(H_f)_{\sing}$ is a finite connected scheme of degree $1$ or $2$, the latter occurring exactly when $\Char k=2$ and $n$ is odd.
Since $(H_f)_{\sing}$ is the fiber of $\varphi$ above $f \in D(k)$,
the general fiber of $\varphi$ is described as in the previous sentence.
The scheme $\calH_{\sing,k}$ is smooth over $k$, hence irreducible,
and its image under $\varphi$ is topologically $D_k$,
so $D_k$ is irreducible too.
The result follows from the previous two sentences.
\end{proof}

Let $\Dfinite \colonequals \{d \in D: \dim \varphi^{-1}(d) = 0\}$; this is the
subset of points such that the corresponding hypersurface has finitely many singular points.
Let $D_1 \colonequals \{d \in D : \varphi^{-1}(d) \to \{d\} \textup{ is an isomorphism}\}$;
this is the subset of points such that the singular locus of the corresponding hypersurface is a single reduced point.

\begin{lemma} \label{L:normalization of Dprime}
  \hfill
  \begin{enumerate}[\upshape (a)]
  \item \label{I:D1 and Dfinite are open}
    The subsets $D_1 \subset \Dfinite \subset D$ are open in $D$.
    Identify them with open subschemes of $D$.
  \item \label{I:above Dfinite}
    $\varphi^{-1}(\Dfinite) \to \Dfinite$ is the normalization of $\Dfinite$.
    The same holds after base change to any normal noetherian domain $A$, except when $\Char A=2$ and $n$ is odd.
In the exceptional case, $\varphi^{-1}(\DfiniteA)$ is the normalization of $(\DfiniteA)_{\red}$ in the purely inseparable extension $\kappa(\calH_{\sing,A})$ of its function field.
  \item \label{I:above D1}
    $\varphi^{-1}(D_1) \to D_1$ is an isomorphism of schemes over $\Z$.
  \end{enumerate}
\end{lemma}

\begin{proof}
  \hfill
  \begin{enumerate}[\upshape (a)]
  \item 
    By \cite{EGA-IV.III}*{Corollaire~13.1.5}, $\Dfinite$ is open in $D$.
Openness of $D_1$ will follow from the proof of~(c).
\item 
By Proposition~\ref{P:birational}, $\varphi^{-1}(\Dfinite) \to \Dfinite$ is birational.
It is also quasi-finite and proper,
hence finite by Zariski's main theorem \cite{EGA-III.I}*{Corollaire~4.4.11}.
Moreover, $\calH_{\sing}$ is smooth over $\Z$, hence normal.
The previous three sentences imply that $\varphi^{-1}(\Dfinite) \to \Dfinite$
is the normalization of $\Dfinite$.
The same argument works after base change to any normal noetherian domain,
except that in the exceptional case, the function field extension in Proposition~\ref{P:birational} is purely inseparable of degree~$2$ instead of $1$.
\item
Apply the following to $\varphi^{-1}(\Dfinite) \to \Dfinite$:
If $\psi \colon X \to Y$ is a scheme-theoretically-surjective finite morphism of noetherian schemes and $y \in Y$ is such that $\psi^{-1}(y) \isom \{y\}$,
  then there exists an open neighborhood $U \subset Y$ of $y$
  such that $\psi^{-1}(U) \to U$ is an isomorphism.
  To prove this statement, we may assume that $Y=\Spec A$ and $X=\Spec B$, where $A \to B$ is injective;
  then $U$ may be taken as the complement of the support of the $A$-module $B/A$.\qedhere
  \end{enumerate}
\end{proof}

\begin{remark}
  In Corollary~\ref{C:D1=Dsmooth}, we will identify $D_1$ with the smooth
  locus of~$D$.
\end{remark}

\section{Commutative algebra}
\label{S:CA}

A ring extension $R' \supset R$
is called a \defi{weakly unramified extension}
if $R'$ too is a discrete valuation ring
and $\pi$ is also a uniformizer of $R'$.

\begin{lemma}
\label{L:extension of DVR}
Let $R$ be a discrete valuation ring, with residue field $k$.
For any field extension $k' \supset k$,
there exists a weakly unramified extension $R' \supset R$
with residue field $k'$ $($i.e., isomorphic to $k'$ as $k$-algebra$)$.
\end{lemma}

\begin{proof}
If $k'/k$ is generated by one algebraic element,
say a zero of a monic irreducible polynomial $\bar{f} \in k[x]$,
then we may take $R' \colonequals R[x]/\ideal{f}$ for any monic $f \in R[x]$ reducing to $\bar{f}$
\cite{SerreLocalFields1979}*{I.\S6, Proposition~15}.
If $k'/k$ is generated by one transcendental element $t$,
then we may take the localization $R' \colonequals R[t]_{\ideal{\pi}}$
of the (regular) polynomial ring $R[t]$
at the codimension~$1$ prime $(\pi)$;
the residue field of $R'$ is $\Frac(R[t]/\ideal{\pi}) = k(t)$.
The general case follows from Zorn's lemma, using direct limits.
\end{proof}

\begin{lemma}
\label{L:completion and normalization}
Let $A$ be a noetherian local ring.
Let $\Ahat$ be its completion.
Let $B$ be the integral closure of $A_{\red}$ $($in its fraction field$)$.
Then
\[
	\#\{\textup{minimal primes of $\Ahat$}\}
        \ge \#\{\textup{maximal ideals of $B$}\}.
\]
\end{lemma}

\begin{proof}
Combine
\cite{StacksProject}*{\href{http://stacks.math.columbia.edu/tag/0C24}{Tag~0C24}}
and
\cite{StacksProject}*{\href{http://stacks.math.columbia.edu/tag/0C28}{Tag~0C28}(1)}.
\end{proof}

The following is well known; see \cite{StacksProject}*{\href{http://stacks.math.columbia.edu/tag/0BRA}{Tag~0BRA}} for a generalization.

\begin{lemma}
\label{L:Spec and purely inseparable extension}
Let $B$ a normal domain.
Let $L=\Frac B$.
Let $L'/L$ be a purely inseparable extension.
Let $B'$ be the integral closure of $B$ in $B'$.
Then $\Spec B' \to \Spec B$ is a homeomorphism.
In particular, $B$ and $B'$ have the same number of maximal ideals.
\end{lemma}

\begin{proof}
We may assume that $p \colonequals \Char L > 0$.
The map $\Spec B \to \Spec B'$ sending $\pp$ to $\{x \in L' : x^{p^e} \in \pp \textup{ for some $e \ge 0$}\}$ is an inverse to $\Spec B' \to \Spec B$.
Thus $\Spec B' \to \Spec B$ is a continuous bijection between quasi-compact spaces, so it is a homeomorphism.
The final sentence follows since maximal ideals correspond to closed points.
\end{proof}

\begin{lemma} \label{L:structure of dvr}
  Let $m \ge 1$.
  Suppose that $\Char(k) = \Char(R)$.
  Then $R/\mm^m \simeq k[t]/\ideal{t^m}$.
  If $R$ is complete, then $R \simeq k\pws{t}$.
\end{lemma}

\begin{proof}
  The ring $R/\mm^m$ is (trivially) a complete local ring as defined
  in~\cite{Cohen1946}. We have $\Char(k) = \Char(R/\mm^m)$,
  so by~\cite{Cohen1946}*{Thm.~9}, $k$ embeds into~$R/\mm^m$.
  Then the surjective homomorphism $k[t] \to R/\mm^m$ mapping~$t$ to~$\pi$ has
  kernel~$\ideal{t^m}$.
  Taking inverse limits gives $R \isom k\pws{t}$ if $R$ is complete.
\end{proof}

\section{Hypersurfaces with several singularities}
\label{S:several}

Let $0 \neq f \in R[x_0,\ldots,x_n]_d$ and set $H = \Proj(R[x_0,\ldots,x_n]/\ideal{f})$.

\begin{theorem}
\label{T:r isolated singularities}
If the space $(H_k)_{\sing}$ consists of $r$ closed points,
then $v(\Delta(f)) \ge r$.
\end{theorem}

\begin{proof}
The inequality is trivial if $r=0$, so assume $r>0$.

Let $P \in D_R(k)$ correspond to $H_k$,
so $\varphi^{-1}(P) = (H_k)_{\sing}$.
Since $R$ is regular,
the local rings $\OO_{\Aff^N_R,P}$ and $\OOhat_{\Aff^N_R,P}$ are regular too,
and hence factorial \cite{Auslander-Buchsbaum1959}*{Theorem~5}.
Since $\dim (H_k)_\sing = 0$, we have $P \in \DfiniteR(k)$ (notation as
in Section~\ref{S:general singular locus}).
Let $L=\kappa(D_R)$ and $L' = \kappa(\calH_{\sing,R})$.
By Lemma~\ref{L:normalization of Dprime}, $\DfiniteR$ is open in $D_R$,
and $\varphi^{-1}(\DfiniteR) \to (\DfiniteR)_{\red}$ is the normalization of $(\DfiniteR)_{\red}$ in the purely inseparable extension $L'/L$.

Localizing at $P$ on the target,
we obtain a morphism $\Spec B' \to \Spec A_{\red}$,
where $A \colonequals \OO_{\DfiniteR,P} = \OO_{D_R,P} = \OO_{\Aff^N_R,P}/\ideal{\Delta}$, 
and $B'$ is the integral closure of $A_{\red}$ in $L'$.
Define $\Ahat$ and $B$ as in Lemma~\ref{L:completion and normalization},
so $\Ahat \isom \OOhat_{\Aff^N_R,P}/\ideal{\Delta}$,
and $B$ is the integral closure of $A_{\red}$ in $L$.
The maximal ideals of $B'$
correspond to the points of $\varphi^{-1}(\DfiniteR)$ above $P$,
which are the $r$ points of $(H_k)_{\sing}$.
By Lemma~\ref{L:Spec and purely inseparable extension}, $B$ too has $r$ maximal ideals.
By Lemma~\ref{L:completion and normalization}, $\Ahat$ has at least $r$ minimal primes.
Their inverse images in $\OOhat_{\Aff^N_R,P}$
correspond to prime factors of $\Delta$ in this factorial ring,
so $\Delta = p_1 \cdots p_r q$,
for some $p_1,\ldots,p_r, q\in \OOhat_{\Aff^N_R,P}$
with each $p_i$ vanishing at $P$.
Evaluation at the coefficient tuple of $f$
defines a ring homomorphism $\OOhat_{\Aff^N_R,P} \to \Rhat$
sending $\Delta$ to $\Delta(f)$
and sending each $p_i$ into the maximal ideal of $\Rhat$,
so $v(\Delta(f)) \ge 1 + \cdots + 1 + 0 = r$.
\end{proof}

\section{Valuations of polynomial values}
\label{S:valuations}

\begin{lemma}
\label{L:projection}
Let $\rho \colon \Aff^\ell_k \to \Aff^n_k$ be a projection for some $\ell \ge n$.
Let $V \subset \Aff^\ell_k$ be a closed subscheme.
Then $\{a \in \Aff^n_k : \rho^{-1}(a) \subseteq V \}$ is closed in $\Aff^n_k$.
\end{lemma}

\begin{proof}
  Since $\rho$ is flat, $\rho$ is open, so $\rho(\Aff^n_k - V)$ is open;
  its complement is closed.
\end{proof}

\begin{definition}
  Let $H = \Spec(k[x_1, \ldots, x_n]/\ideal{f}) \subset \Aff^n_k$ be a hypersurface
  and let $a \in k^n$.
  Let $\mm_a$ be the maximal ideal of~$k[x_1, \ldots, x_n]$ corresponding to~$a$.
  Then $\mult_H(a)$ denotes the \defi{multiplicity} of~$a$ as a point on~$H$, i.e.,
  \[ \mult_H(a) = \max\{m \in \Z_{\ge 0} : f \in \mm_a^m\}. \]
\end{definition}

For $b \in R$, let $\bar{b}$ be its image in $k$.
Likewise, given $b \in R^n$, define $\bar{b} \in k^n$.
Fix a nonzero polynomial $\delta \in R[x_0,\ldots,x_n]$.
(Eventually $\delta$ will be $\Delta$.)
{}From now on, we assume that $k$ is infinite.

\begin{definition}
Define $\vmin_\delta \colon k^n \to \Z_{\ge 0} \union \{\infty\}$ by
\[
	\vmin_\delta(a) = \min\{v(\delta(b)) : b \in R^n \text{\ with\ } \bar{b} = a\}.
\]
\end{definition}

\begin{lemma}
  \label{L:Gauss norm}
  The integer $\min \{v(\delta(b)) : b \in R^n\}$
  equals the minimum of the valuations of the coefficients of $\delta$.
\end{lemma}

\begin{proof}
By dividing by a power of $\pi$, we may assume that some coefficient is a unit.
The reduction $\bar{\delta} \in k[x_1,\ldots,x_n]$
is not the zero polynomial, and $k$ is infinite,
so $\bar{\delta}$ is nonvanishing at some point in $k^n$.
Lift the point to $b \in R^n$.
Then $v(\delta(b))=0$.
Thus both minima equal~$0$.
\end{proof}

\begin{corollary}
  \label{C:formula for vmin}
  For $b \in R^n$, the integer $\vmin_{\delta}(\bar{b})$
 equals the minimum of the valuations of the coefficients of $\delta(b+\pi x)$.
\end{corollary}

\begin{proof}
  Apply Lemma~\ref{L:Gauss norm} to $\delta(b+\pi x)$.
\end{proof}

\begin{proposition}
\label{P:Greenberg}
The function $\vmin_\delta$ on $\Aff^n(k)$ is upper-semicontinuous with respect to the Zariski topology.
\end{proposition}

\begin{proof}
  We need to show that for $m \in \Z_{\ge 0}$,
  the set $\{a \in k^n : \vmin_\delta(a) \ge m \}$
  is $W(k)$ for some closed subscheme $W \subset \Aff^n_k$.
  Let $R_m = R/\mm^m$.

\emph{Case 1: $R$ is of equal characteristic.}
By Lemma~\ref{L:structure of dvr}, $R_m$ is a $k$-algebra
of vector space dimension $m$.
Applying restriction of scalars $\Res_{R_m/k}$ to
$\delta \colon \Aff^n_{R_m} \to \Aff^1_{R_m}$
produces a morphism $\Aff^{mn}_k \to \Aff^m_k$;
let $V$ be the fiber above~$0$.
The reduction map ${R_m}^n \to k^n$ arises from a morphism
$\rho \colon \Aff^{mn}_k \to \Aff^n_k$
that is a projection as in Lemma~\ref{L:projection}.
Let $W$ be a closed subscheme whose underlying space
is the closed subset of Lemma~\ref{L:projection}.
Then for $a \in k^n$, the following are equivalent
(note that $\rho^{-1}(a)$ is an affine space):
\[
	\vmin_{\delta}(a) \ge m, \qquad
	\rho^{-1}(a)(k) \subset V(k), \qquad
        \rho^{-1}(a) \subset V, \qquad
        a \in W(k).
\]

\emph{Case 2: $R$ is of mixed characteristic with perfect residue field $k$.}
In the previous argument,
replace $\Res_{R_m/k}$
with the Greenberg functor $\Gr^m$ from $R_m$-schemes to $k$-schemes;
see \citelist{\cite{Greenberg1961} \cite{Greenberg1963} \cite{Nicaise-Sebag2008}*{\S2.2} \cite{Bertapelle-Gonzalez-Aviles2018}}.

\emph{Case 3: $R$ is of mixed characteristic with imperfect residue field $k$.}
Let $k'$ be the perfect closure of $k$.
Use Lemma~\ref{L:extension of DVR} to find a weakly unramified extension
$R' \supset R$ with residue field $k'$.
Let $W' = \Spec(k'[x_1,\ldots,x_n]/\ideal{f_1,\ldots,f_r})$
be the closed subscheme for $R'$ in Case~2.
By replacing each $f_i$ by $f_i^{p^n}$ for some $n$,
we may assume that $f_i \in k[x_1,\ldots,x_n]$,
without changing $W'(k')$.
Let $W = \Spec(k[x_1,\ldots,x_n]/\ideal{f_1,\ldots,f_r})$.
By Corollary~\ref{C:formula for vmin},
$\vmin_\delta(a)$ is the same whether we work with $R$ or $R'$,
so $\{a \in k^n : \vmin_\delta(a) \ge m\} = k^n \intersect W'(k') = W(k)$.
\end{proof}

Let $V = \Spec(R[x_1,\ldots,x_n]/\ideal{\delta})$.
{}From now on, assume that some coefficient of $\delta$ is a unit,
so that $V_k$ is a hypersurface in $\Aff^n_k$.

\begin{lemma} \label{L:vmin and multiplicity}
  Let $a \in k^n$.
  Then $\vmin_\delta(a) \le \mult_{V_k}(a)$.
\end{lemma}

\begin{proof}
  Without loss of generality, $a=0$.
  Let $m=\mult_{V_k}(0)$.
  Some degree $m$ monomial in $\delta(x)$ has a unit coefficient,
  so some degree $m$ monomial in $\delta(\pi x)$ has valuation $m$.
  On the other hand, $\vmin_\delta(0)$ is the minimum of the valuations of $\delta(\pi x)$, so it is at most $m$.
\end{proof}

\begin{proposition} \label{P:Vk2}
  Let $a \in k^n$.
  Then $\vmin_\delta(a) \ge 2$ if and only if $a \in (V_k)_{\sing}$ and $a$ is in the image of the reduction map $V(R/\mm^2) \to V(k)$.
\end{proposition}

\begin{proof}
  By shifting, we may assume $a=0$.
  Write $\delta(x) = r + \sum_{i=1}^n s_i x_i + \dots$.
The following are equivalent:
\begin{itemize}
\item $\vmin_\delta(0) \ge 2$;
\item the minimum of the valuations of the coefficients of $\delta(\pi x)$ is at least $2$ \\
  (see Corollary~\ref{C:formula for vmin});
\item $v(r) \ge 2$ and $v(s_i) \ge 1$ for all $i$.
\end{itemize}
The last conditions imply that $0 \in V(R/\mm^2)$ and $0 \in (V_k)_{\sing}$.
Conversely, if $0 \in (V_k)_\sing$,
then $v(r) \ge 1$ and $v(s_i) \ge 1$ for all $i$,
and if moreover $0$ is the image of some $b_2 \in V(R/\mm^2)$,
then we may lift $b_2$ to $b \in (\pi R)^n$ with $v(\delta(b)) \ge 2$,
which is equivalent to $v(r) \ge 2$ since $v(s_i) \ge 1$ for all $i$.
\end{proof}

\section{Minimal valuations of the discriminant}
\label{S:vminDelta}

In this section, we assume that $k$ is algebraically closed.
Recall the definitions of $\Delta$ and $D$ from Section~\ref{S:discriminant}.
We apply the results of the previous section with $\delta \colonequals \Delta$.
For $a \in k^N$, let $f_a \in k[x_0,\ldots,x_n]_d$ be the polynomial with coefficients given by $a$, and let $H_a = H_{f_a} \subset \PP^n_k$.

By Theorem~\ref{T:r isolated singularities},
if $(H_a)_\sing$ consists of $r$ isolated points,
then $\vmin_\Delta(a) \ge r$.

\begin{lemma} \label{L:vmin of Delta}
  Fix $b \in R^N$.
  \begin{enumerate}[\upshape(a)]
    \item \label{I:vmin of Delta,1}
          Let $V \subset \Aff^N_k$ be a variety such that $\bar{b} \in V(k)$.
          If $\{a \in V(k) : \vmin_\Delta(a) \ge m \}$ is Zariski dense in~$V$, then $v(\Delta(b)) \ge m$.
    \item \label{I:vmin of Delta,2}
          If there exists $a \in k^N$ such that $(H_a)_\sing$ is finite and contains $r$ distinct points $P_1, \ldots, P_r$
          that are also singularities of~$H_{\bar{b}}$, then $v(\Delta(b)) \ge r$.
  \end{enumerate}
\end{lemma}

\begin{proof} \strut
  \begin{enumerate}[\upshape(a)]
  \item By Proposition~\ref{P:Greenberg}, $\{a \in V(k) : \vmin_\Delta(a) \ge m \} = V(k) \ni \bar{b}$, so $v(\Delta(b)) \ge m$.
  \item If $\bar{b}=a$, then $v(\Delta(b)) \ge \vmin_{\Delta}(a) \ge r$ by Theorem~\ref{T:r isolated singularities}.
    If $\bar{b} \ne a$, let $V \subset \Aff^N_k$ be the line joining $\bar{b}$ and~$a$.
Since the condition that a given point
$P \in \PP^{n}$ is a singular point of a hypersurface~$H$ is linear in
the coefficients of the polynomial defining~$H$, all points $c \in V(k)$
will have the property that $\{P_1, \ldots, P_r\} \subset (H_c)_\sing$.

By Lemma~\ref{L:normalization of Dprime}\eqref{I:D1 and Dfinite are open},
``$\dim (H_c)_\sing \ge 1$'' is a closed condition,
so for all but finitely many $c \in V(k)$,
we have $\dim (H_c)_\sing = 0$, so $(H_c)_{\sing}$ is a finite set containing $P_1,\ldots,P_r$.
Theorem~\ref{T:r isolated singularities} implies that $\vmin_\Delta(c) \ge r$ for all these points.
The claim now follows from part~\eqref{I:vmin of Delta,1}.  \qedhere
  \end{enumerate}
\end{proof}

\begin{lemma} \label{L:surjective reduction on D}
  The reduction map $D(R) \to D(k)$ is surjective.
\end{lemma}

\begin{proof}
  Let $a \in D(k)$; then $H_a$ has a singular point~$P \in \PP^n(k)$.
  We may assume that $P=(1:0:\cdots:0)$.
  The condition that $H_a$ is singular at~$P$ is given by the vanishing
  of certain coordinates. Lift $a \in k^N$ to some~$b \in R^N$
  so that these coordinates remain zero.
\end{proof}

\begin{corollary} \label{C:description of D2}
 Let $a \in k^N$.
  Then $\vmin_\Delta(a) \ge 2$ if and only if $a \in (D_k)_{\sing}$.
\end{corollary}

\begin{proof}
By Lemma~\ref{L:surjective reduction on D}, every $a \in D(k)$ is in the image of~$D(R/\mm^2)$.
Apply Proposition~\ref{P:Vk2} to $\delta = \Delta$.
\end{proof}

We now prove a variant of Theorem~\ref{T:r isolated singularities},
in which the $r$ singularities need not be isolated,
but they must be linearly independent.

\begin{lemma} \label{L:mult 2 if sing at 2 points}
  Let $P_1, \ldots, P_r \in \PP^n(k)$ be points that span a~$\PP^{r-1}$ and let
  $a_0 \in k^N$ be such that $H_{a_0}$ is singular at $P_1, \ldots, P_r$.
  Then $\vmin_\Delta(a_0) \ge r$ and $\mult_{D_k}(a_0) \ge r$.
\end{lemma}

\begin{proof}
  We can assume the $P_j$ to be coordinate points.
  Since $\Delta$ vanishes on $a \in k^N$ when $H_a$ is singular in~$P_j$,
  each term in~$\Delta$ must be divisible by one of the coordinates that
  describe the vanishing of $a$ and its first partial derivatives at~$P_j$.
  When the degree~$d$ is at least~$3$, then these
  sets of coordinates are disjoint in pairs, and so
  \[ \Delta  \in \ideal{a(P_1), \nabla a(P_1)} \cdots \ideal{a(P_r), \nabla a(P_r)} . \]
  This implies that $\mult_{D_k}(a_0) \ge r$ and also that $\vmin_\Delta(a_0) \ge r$.

  The result is still valid when $d = 2$.
  In this case, the associated reduced subscheme of $(H_{a_0})_{\sing}$ is a linear space, so the $\PP^{r-1}$ spanned by $P_1,\ldots,P_r$ is contained in $(H_{a_0})_{\sing}$.
Then by Proposition~\ref{P:quadratic form}\eqref{I:inequality for quadratic form},
$\vmin_\Delta(a_0) \ge r$.
By Lemma~\ref{L:vmin and multiplicity},
 $\mult_{D_k}(a_0) \ge \vmin_\Delta(a_0) \ge r$.
\end{proof}

Lemma~\ref{L:mult 2 if sing at 2 points} generalizes 
Proposition~\ref{P:quadratic form}\eqref{I:inequality for quadratic form}
to forms of arbitrary degree.

For a subset $X \subset \PP^n(k)$, let $\Span X$ be the smallest linear subspace of $\PP^n$ containing $X$.
If $X=\emptyset$, use the conventions $\Span X \isom \PP^{-1} = \emptyset$ and $\dim \Span X = -1$.

\begin{corollary} \label{C:lower bound in terms of dim sing}
  Let $b \in R^N$. Then $v(\Delta(b)) \ge \dim \Span((H_{\bar b})_\sing(k)) + 1$.
\end{corollary}

\begin{proof}
  Let $r = \dim \Span((H_{\bar b})_\sing(k)) + 1$.
  Choose $P_1,\ldots,P_r \in (H_{\bar b})_\sing(k)$ that span a~$\PP^{r-1}$.
  Now apply Lemma~\ref{L:mult 2 if sing at 2 points}.
\end{proof}

We will obtain better lower bounds in Section~\ref{S:positive-dimensional}.

\begin{corollary} \label{C:vDelta=1 then one singularity}
  Let $b \in R^N$ such that $v(\Delta(b)) = 1$. Then $(H_{\bar b})_\sing$
  consists of a single point.
\end{corollary}

\begin{proof}
  Since $\Delta(\bar{b}) = \overline{\Delta(b)} = 0$, $H_{\bar b}$ has at
  least one singularity.
  If $(H_{\bar b})_\sing$ contained at least two points, then 
  $v(\Delta(b)) \ge \vmin_\Delta(\bar{b}) \ge 2$ by Lemma~\ref{L:mult 2 if sing at 2 points}, contradicting the assumption.
\end{proof}

See Corollary~\ref{C:vmin=1 and nondeg double pt} below for a more
precise statement.

\section{When the discriminant has valuation 1}
\label{S:valuation 1}

We now characterize when $v(\Delta(f)) = 1$. Recall the statement
from the introduction:

\begin{theorem1.1}
Let $f \in R[x_0,\dots,x_n]$ be a homogeneous polynomial.
Let $\Delta(f)$ be its discriminant.
Let $H = \Proj(R[x_0,\dots,x_n]/\ideal{f})$.
Then the following are equivalent:
\begin{enumerate}[\upshape (i)]
\item $v(\Delta(f))=1$;
\item $H$ is regular,
and $(H_k)_{\sing}$ consists of a single nondegenerate double point in $H(k)$.
\end{enumerate}
\end{theorem1.1}

\begin{proof}
\emph{Case 1: $\Char k=2$ and $n$ is odd.}
By \cite{Saito2012}*{Theorem~4.2},
if the sign of $\Delta$ is chosen appropriately,
then $\Delta = A^2+4B$ for some polynomials $A,B$,
so $v(\Delta(f)) \ne 1$.
On the other hand, by Remark~\ref{R:explicit nondegenerate double point},
$H_k$ cannot have a nondegenerate double point.
Thus (i) and~(ii) both fail.

\emph{Case 2: $\Char k \ne 2$ or $n$ is even.}
The surjection $R[\{a_\ii\}] \surjects R$
sending the $a_\ii$ to the corresponding coefficients $\alpha_\ii$ of $f$
defines an $R$-morphism $\iota \colon \Spec R \to \Aff^N_R$.
Then $H \to \Spec R$ is the pullback by $\iota$ of $\calH_R \to \Aff^N_R$.
Let $P = \iota(\Spec k) \in \Aff^N(k)$.

(i)$\Rightarrow$(ii): Suppose that $v(\Delta(f))=1$.
By Corollary~\ref{C:vDelta=1 then one singularity},
$(H_k)_{\sing}$ consists of a single point.
The surjection $R[\{a_\ii\}] \surjects R$ maps $\Delta$ to $\Delta(f)$,
so the $a_\ii - \alpha_\ii$ and $\Delta$ are local parameters
for $\Aff^N_R$ at $P$.
Thus $D_R = \Spec(R[\{a_\ii\}]/\ideal{\Delta})$ is regular at $P$,
so $D_R$ is normal at $P$.
Then Lemma~\ref{L:normalization of Dprime}\eqref{I:above Dfinite} implies that
the fiber $(H_k)_{\sing} = \varphi^{-1}(P)$
consists of a single reduced $k$-point~$Q$.
By Remark~\ref{R:nondegenerate double point},
$Q$ is a nondegenerate double point of~$H_k$.

Choose an $\Aff^n_R \subset \PP^n_R$ containing $Q$;
let $f_0$ be the corresponding dehomogenization of $f$.
The point $(H_k)_{\sing}$ is cut out in $\Aff^n_R$
by $f_0$ and its partial derivatives;
these $n+1$ functions are therefore local parameters for $\PP^n_R$ at $Q$,
so the local ring $\OO_{H,Q} = \OO_{\PP^n_R,Q}/\ideal{f_0}$ is regular too.
On the other hand, $H-\{Q\}$ is smooth over $\Spec R$.
Thus $H$ is regular everywhere.

(ii)$\Rightarrow$(i):
Now suppose that $H$ is regular
and $(H_k)_{\sing}$ consists of a nondegenerate double point $Q \in H(k)$.
Let $f_0$ be as above,
so $f_0$ and its partial derivatives lie in the maximal ideal
$\mm_{\PP^n_R,Q} \subset \OO_{\PP^n_R,Q}$.
Since $Q$ is a nondegenerate double point, the partial
derivatives form a basis for $\mm_{\PP^n_k,Q}/\mm_{\PP^n_k,Q}^2$,
so they are independent in $\mm_{\PP^n_R,Q}/\mm_{\PP^n_R,Q}^2$.
On the other hand, the image of $f_0$ in $\mm_{\PP^n_R,Q}/\mm_{\PP^n_R,Q}^2$
is nonzero (since $\OO_{H,Q} = \OO_{\PP^n_R,Q}/\ideal{f_0}$ is regular)
and in fact \emph{independent} of the partial derivatives
(since it maps to $0$ in $\mm_{\PP^n_k,Q}/\mm_{\PP^n_k,Q}^2$).
Thus $f_0$ and its partial derivatives
form a basis of $\mm_{\PP^n_R,Q}/\mm_{\PP^n_R,Q}^2$,
so by Nakayama's lemma, they generate $\mm_{\PP^n_R,Q}$,
so $H_{\sing} \isom \Spec k$.

Pulling back $(\calH_R)_{\sing} \to D_R \injects \Aff^N_R$
by $\iota$ gives $H_{\sing} \to \Spec(R/\ideal{\Delta(f)}) \to \Spec R$.
Since $(H_k)_{\sing}$ is a single reduced $k$-point, $P \in D_1(k)$.
By Lemma~\ref{L:normalization of Dprime}\eqref{I:D1 and Dfinite are open},
$D_{1,R}$ is open in $D_R$,
so $\Spec(R/\ideal{\Delta(f)})$ is contained in $D_{1,R}$.
By Lemma~\ref{L:normalization of Dprime}\eqref{I:above D1},
$(\calH_R)_{\sing} \to D_R$ is an isomorphism above $D_{1,R}$,
so $H_{\sing} \isom \Spec(R/\ideal{\Delta(f)})$.
By the previous paragraph, $H_{\sing} \isom \Spec k$,
so $v(\Delta(f))=1$.
\end{proof}

\begin{corollary} \label{C:vmin=1 and nondeg double pt}
  Assume that $k$ is algebraically closed. For $a \in k^N$
  $($and for every choice of~$R$ with residue field~$k$$)$, the following statements are equivalent.
  \begin{enumerate}[\upshape(a)]
    \item \label{I:vmin=1,a}
          $\vmin_\Delta(a) = 1$.
    \item \label{I:vmin=1,b}
          $a$ is a smooth point on~$D_k$.
    \item \label{I:vmin=1,c}
      $(H_a)_\sing$ consists of a single nondegenerate double point.
    \item \label{I:vmin=1,d}
      $a \in D_1(k)$.
  \end{enumerate}
\end{corollary}

\begin{proof}\hfill

  \eqref{I:vmin=1,a}$\Leftrightarrow$\eqref{I:vmin=1,b}:
  The following are equivalent:
  $\vmin_\Delta(a)>0$; $\Delta(a)=0$ in $k$; $a \in D(k)$.
  By Corollary~\ref{C:description of D2}, $\vmin_\Delta(a) \ge 2$ if and only if $a \in D_{\sing}(k)$.

  \eqref{I:vmin=1,a}$\Rightarrow$\eqref{I:vmin=1,c}: Use Theorem~\ref{T:valuation 1}.

  \eqref{I:vmin=1,c}$\Rightarrow$\eqref{I:vmin=1,a}:
  Let $a \in k^N$ be such that $(H_a)_\sing$ consists of a single nondegenerate double point~$Q$.
  Lift $a$ to $b \in R^N$.
  Then $H_b$ is regular at every point of its special fiber except possibly $Q$.
  By adding a multiple of $\pi$ to $b$ if necessary,
  we may assume that $H_b$ is regular also at $Q$.
  The regular locus of $H_b$ is open and contains the special fiber,
  so $H_b$ is regular.
Theorem~\ref{T:valuation 1} applied to $H_b$ implies that $v(\Delta(b))=1$.
On the other hand, if $b'$ is any lift of $a$,
then $v(\Delta(b')) \ge 1$ since $H_a$ is singular.
Thus $\vmin_\Delta(a)=1$.

\eqref{I:vmin=1,c}$\Leftrightarrow$\eqref{I:vmin=1,d}:
Use Remark~\ref{R:nondegenerate double point} and the definition of $D_1$.
\end{proof}

\begin{corollary} \label{C:D1=Dsmooth}
The subscheme $D_1$ is the smooth locus of $D \to \Spec \Z$.
\end{corollary}

\begin{proof}
  This can be checked on geometric points,
  and every field $k$ is the residue field of some discrete valuation ring $R$.
  Apply Corollary~\ref{C:vmin=1 and nondeg double pt}\eqref{I:vmin=1,b}$\Leftrightarrow$\eqref{I:vmin=1,d}.
\end{proof}

\section{Hypersurfaces with a positive-dimensional singularity}
\label{S:positive-dimensional}

In Lemma~\ref{L:surjection on global sections},
Corollary~\ref{C:dimension of locus of r-singular hypersurfaces},
and Lemma~\ref{L:dense},
we assume that $n \ge 2$, $r \ge 1$, and
$P_1,\ldots,P_r$ are distinct points in $\PP^n(k)$.
Let $\OO = \OO_{\PP^n_k}$.
For each $P \in \PP^n(k)$, let $\mm_P \subset \OO$ be the ideal sheaf of $P$.

\begin{lemma}
\label{L:surjection on global sections}
If $d \ge 2r-1$, then $\OO(d) \to \prod_i (\OO/\mm_{P_i}^2)(d)$
induces a surjection on global sections.
\end{lemma}

\begin{proof}
Surjectivity of a linear map is unchanged by field extension,
so we may assume that $k$ is infinite. Then we can choose, for
each $1 \le i \le r$, a linear form~ $\ell_i$ vanishing at~$P_i$
but not at~$P_j$ for any $j \ne i$.
We can also choose a homogeneous polynomial~$h$ of degree $d-(2r-1)$
not vanishing at any $P_i$.
For each $s$, as $g$ ranges over linear forms,
the image of $g$ in $(\OO/\mm_{P_s}^2)(1)$ ranges over all its sections,
so the images of $g h \prod_{j \ne s} \ell_j^2$
in $\prod_i (\OO/\mm_{P_i}^2)(d)$
exhaust the $s$th factor of $\prod_i (\OO/\mm_{P_i}^2)(d)$.
\end{proof}

\begin{remark}
The result of Lemma~\ref{L:surjection on global sections} is sharp
when the points $P_1, \ldots, P_r$ are on a line: in this case,
no $d < 2r-1$ will have the stated property.
\end{remark}

Recall the definitions of $N$ and $H_f$ from Section~\ref{S:discriminant}.
Let $Z \subset \Aff^N$ be the subvariety whose points correspond to
$f$ such that $(H_f)_{\sing}$ contains $P_1,\ldots,P_r$.

\begin{corollary}
\label{C:dimension of locus of r-singular hypersurfaces}
If $d \ge 2r-1$, then $Z$ is an affine space of dimension $N-r(n+1)$.
\end{corollary}

\begin{proof}
The set $Z(k)$ is the kernel of the surjection in Lemma~\ref{L:surjection on global sections}.
\end{proof}

\begin{lemma}
\label{L:dense}
Assume that $k$ is algebraically closed.
If $r \le (d-1)/2$, then there exists $f \in k[x_0,\ldots,x_n]_d$
such that $(H_f)_{\sing} = \{P_1,\ldots,P_r\}$ as a set.
\end{lemma}

\begin{proof}
Let $P = \{P_1,\ldots,P_r\}$.
Let
\[
	I = \{(f,P_{r+1}) \in Z \times (\PP^n-P)  :  P_{r+1} \in (H_f)_{\sing} \}.
\]
The fiber of $I \to \PP^n - P$ above $P_{r+1}$
consists of the $f$ for which $(H_f)_{\sing} \supset \{P_1,\ldots,P_{r+1}\}$,
so its dimension is $N-(r+1)(n+1)$
by Corollary~\ref{C:dimension of locus of r-singular hypersurfaces},
which also implies $\dim Z = N - r(n+1)$.
Thus $\dim I = \dim(\PP^n-P) + N-(r+1)(n+1) = \dim Z - 1$.
Therefore there exists a point in $Z$ outside the image of $I$;
this defines $f$.
\end{proof}

\begin{theorem}
  \label{T:positive-dimensional singularity}
Let $n,d \ge 2$.
Let $f \in R[x_0,\ldots,x_n]_d$.
Let $H = H_f$.
Assume that $\dim \, (H_k)_{\sing} \ge 1$.
\begin{enumerate}[\upshape(a)]
  \item \label{I:positive-dimensional singularity,a}
        $v(\Delta(f)) \ge \dim \, (H_k)_{\sing} + 1 \ge 2$.
  \item \label{I:positive-dimensional singularity,b}
        $v(\Delta(f)) \ge \lfloor(d-1)/2\rfloor$.
  \item \label{I:positive-dimensional singularity,c}
        If $n = 2$, then $v(\Delta(f)) \ge 2d-3$ if $d \neq 4$
        and $v(\Delta(f)) \ge 4$ if $d = 4$.
  \item \label{I:positive-dimensional singularity,d}
        If $(H_k)_\sing$ contains a line, then $v(\Delta(f)) \ge d-1$.
\end{enumerate}
\end{theorem}

\begin{proof}
Using Lemma~\ref{L:extension of DVR},
we may reduce to the case in which $k$ is algebraically closed.
\begin{enumerate}[\upshape(a)]
  \item This follows from Corollary~\ref{C:lower bound in terms of dim sing}.
  \item Let $r = \lfloor(d-1)/2\rfloor$.
        Choose distinct points $P_1, \ldots, P_r \in (H_k)_\sing$.
        By Lemma~\ref{L:dense}, there exists $h \in k[x_0,\ldots,x_n]_d$ such that $(H_h)_\sing = \{P_1, \ldots, P_r\}$ as a set.
        By Lemma~\ref{L:vmin of Delta}\eqref{I:vmin of Delta,2},
        $v(\Delta(f)) \ge r$ as claimed.
      \item In the case $n = 2$ of plane curves, $\bar{f} = g^2 h$
for some $g$ of some degree~$m$ with $1 \le m \le d/2$ and $h$ of degree~$d-2m$.
        In Lemma~\ref{L:vmin of Delta}\eqref{I:vmin of Delta,1} we take
        $V$ to be the variety consisting of all forms factoring as $g_1 g_2 h$
        with $\deg g_1 = \deg g_2 = m$ and $\deg h = d-2m$; then $\bar{f} \in V(k)$.
        Let $V' \subset V$ be the dense open subvariety defined by the additional conditions that $g_1,g_2,h$ define smooth curves intersecting transversely.
        By B\'ezout's theorem, if $a \in V'(k)$, then $\#(H_a)_{\sing} = m^2 + 2m(d-2m)$, so $\vmin_\Delta(a) \ge m^2 + 2m(d-2m)$ by Theorem~\ref{T:r isolated singularities}.
        Lemma~\ref{L:vmin of Delta}\eqref{I:vmin of Delta,1} then shows
        that $v(\Delta(f)) \ge m^2 + 2m(d-2m)$.
        The bound in the statement
        is obtained by taking the minimum over $m$ in $[1,d/2]$.
        For $d \neq 4$,
        the minimum is obtained for $m = 1$; when $d = 4$, $m = 2$ gives the
        smaller value.
  \item Let $L$ be a line contained in~$(H_k)_{\sing}$. We will construct an auxiliary polynomial $h \in k[x_0,\ldots,x_n]_d$ such that $(H_h)_{\sing}$ is finite 
    and contains $d-1$ points on~$L$.
        We may assume that $L$ is $x_2 = x_3 = \dots = x_n = 0$.
        Choose distinct $c_1,\ldots,c_{d-1} \in k$.
        Let $g \in k[x_3,\ldots,x_n]_d$ such that $H_g \subset \PP^{n-3}$ is smooth; if $n=2$, then $g=0$.
        Let $h = x_2 \prod_{i=1}^{d-1} (x_1 - c_i x_0) + g(x_3, \ldots, x_n)$.

        Suppose that $Q \in (H_h)_{\sing}$.
        At $Q$, we have $\del h/\del x_2 = 0$ so $\prod_{i=1}^{d-1} (x_1 - c_i x_0) = 0$; also $h=0$, so $g=0$; also, $\del g/\del x_i = 0$ for $i=3,\ldots,n$, but $H_g$ is smooth.
        Thus $x_3=\cdots=x_n=0$ at $Q$, and $Q$ is a singular point on the union of lines $x_2 \prod_{i=1}^{d-1} (x_1 - c_i x_0)=0$ in $\PP^2$,
        hence $(0:0:1)$ or $(1:c_i:0)$ for some $i$.
        Thus $(H_h)_{\sing}$ is finite and contains $d-1$ points on $L$.
        Lemma~\ref{L:vmin of Delta}\eqref{I:vmin of Delta,2} gives $v(\Delta(f)) \ge d-1$.
  \qedhere
\end{enumerate}
\end{proof}

\section*{Acknowledgments}

We thank Parimala and Jean-Pierre Tignol for providing information
about quadratic forms over discrete valuation rings
of residue characteristic~$2$.

\end{document}